\documentclass[12pt]{article}
\usepackage{amsfonts}
\usepackage{eurosym}
\usepackage{amssymb}
\usepackage{amsmath,amssymb,latexsym}
\usepackage[dvips]{graphics}
\newtheorem{theorem}{Theorem}[section]

\newtheorem{proof}[theorem]{Proof}
\numberwithin{equation}{section}

\begin{document}

\title{\textbf{The Bakry-\'{E}mery Ricci tensor: Application to mass distribution in space-time}}
\author{{Ghodratallah Fasihi-Ramandi$\sp{a}$\thanks{fasihi@sci.ikiu.ac.ir (Gh. Fasihi Ramandi)}}\\
{\small \textit{Department of Mathematics, Faculty of Science, Imam Khomeini International University, Qazvin, Iran}}}
\date{}
\maketitle
\begin{abstract}
The Bakry-\'{E}mery Ricci tensor gives an analogue of the Ricci tensor for a Riemannian manifold with a smooth function. This notion motivates new version of Einstein field equation in which mass becomes part of geometry. This new field equation is purely geometric and is obtained from an action principle which is formed naturally by scalar curvature associated to the Bakry-\'{E}mery Ricci tensor.
\end{abstract}
\section{Introduction}

No well-defined current physical theory claims to model all nature; each intentionally neglects some effects. Roughly, general relativity is a model of nature, especially of gravity, that neglects quantum effects. In fact, general relativity describes gravity in the context of 4-dimensional Lorentzian manifolds by Einstein field equation. Einstein argued that the stress-energy of matter and electromagnetism influences space-time $(M,g)$ and suggested his field equation as follows \cite{sash}.
\begin{equation}\label{einstein}
\mathrm{Ric}-\dfrac{1}{2}Rg=T+E,
\end{equation}
where, $T$ and $E$ are stress-energy tensor of matter and electromagnetic field, respectively. The expressions in the right hand side of the Einstein equation \ref{einstein} are physical concepts whilst  the left hand side is completely geometric. In fact, influence of matter and electromagnetism on space-time is not described directly by the geometry of space-time manifold $(M,g)$. Hence, a question naturally arises here, is:  how to geometrize matter and electromagnetism?\\

 Attempts at geometrization of electromagnetism also known as unification of gravity and electromagnetism have been made ever since the advent of general theory of relativity. Most of these attempts share the idea that Einstein's original theory must in some way be generalized that some part of geometry describes electromagnetism. Nevertheless, there are relatively few results on geometrization of mass in general relativity or, more generally applicable to space-time manifold, and many of them rely on dimension 5 or other geometric structures like Lie algebroids (see \cite{Torsten}, \cite{Amagh} and \cite{mass}).\\
 
In the following, we present some clues which show that replacing the Ricci curvature tensor of a space-time manifold $(M,g)$ by the Bakry-\'{E}mery Ricci tensor provides us a good apparatus for geometrization of mass.\\

(1) Because of the Einstein equation $\mathrm{Ric}-\dfrac{1}{2}Rg=T+E$, any attempt on geometrization of mass and electromagnetism then becomes a condition on the Einstein equation and so, ultimately, on the Ricci curvature. Once one has a condition on Ricci curvature, one has a tool to geometrization of mass.\\
Suppose that $f$ is a smooth function on a Riemannian manifold $(M,g)$, then the $m$-Bakry-\'{E}mery Ricci tensor is defined as
\[\mathrm{Ric}_f^m =\mathrm{Ric}+\mathrm{Hess} f-\dfrac{1}{m}df\otimes df,\]
where, $\mathrm{Ric}$ stands for Ricci curvature tensor of $(M,g)$ and $m$ is a non-zero constant that is also allowed to be infinite, in which case we write $\mathrm{Ric}_f^\infty =\mathrm{Ric}+\mathrm{Hess} f$. The Ricci curvature tensor can describes gravity and the $m$-Bakry-\'{E}mery Ricci tensor is capable of describing gravity and matter, simultaneously. In fact, the Ricci curvature describes gravity and $f$ can play the role of potential for mass distribution in space-time.\\

(2) From the point of view of the foundation of general relativity, projective structures play an important role; the reason relies on the relations 'free falling (massive) particles-pre-geodesics-projective structures' (see \cite{ehl}).\\
Let $\nabla$ be the Levi-Civita connection of a Riemannian manifold $(M,g)$ of dimension $n$. For a 1-form $\alpha$ define 
\[
\nabla_X^\alpha Y= \nabla_X Y-\alpha (X)Y-\alpha(Y)X.\]
The affine connection $\nabla^\alpha$ is torsion-free, moreover it is projectively equivalent to $\nabla$, meaning that $\nabla^\alpha$ has the same geodesics, up to re-parametrization. It can be seen that any torsion-free connection projectively equivalent to $\nabla$ is of the form $\nabla^\alpha$ for some $\alpha$.\\
To see the link to Bakry-\'{E}mery Ricci curvature, let $\alpha=\dfrac{df}{n-1}$, then we recover a Bakry-\'{E}mery Ricci Ricci tensor. As a simple calculation shows that (Proposition 3.3 in \cite{WYLIE}) 
\[ \mathrm{Ric}^{\nabla^\alpha} =\mathrm{Ric}_g +\mathrm{Hess}f +\dfrac{df\otimes df}{n-1}=\mathrm{Ric}_f^{1-n}.
\]
In other words, the Ricci tensor of a projectively equivalent connection is exactly the $(1-n)$-Bakry-\'{E}mery Ricci tensor. The Ricci tensor of a projectively equivalent connection can be seen as a condition on Ricci curvature and so can be a tool for geimetrization of mass.\\

(3) We already know Einstein metrics play a crucial role in general theory of relativity. Hence, the key is provided by generalization of Einstein metrics should probably capable of solve our concern.\\
A triple $(M, g, f)$ (a Riemannian manifold $(M, g)$ with a smooth function $f$ on $M$) is said to be $(m-)$quasi-Einstein if it satisfies the equation
\[\mathrm{Ric}_f^m =\mathrm{Ric}+\mathrm{Hess} f-\dfrac{1}{m}df\otimes df=\lambda g,\]
for some $\lambda\in \mathbb{R}$. The above equation is especially interesting in that when $m=\infty$ it is exactly the gradient Ricci soliton equation; when $m$ is a positive integer, it corresponds to warped product Einstein metrics (see \cite{JEFFREY}); when $f$
is constant, it gives the Einstein equation. We call a quasi-Einstein metric trivial when f is constant. Following the terminology of Ricci solitons, a quasi-Einstein metric on a manifold $M$ will be called expanding, steady or shrinking, respectively, if $\lambda <0$, $\lambda =0$ or $\lambda >0$.\\
As we have noticed, the Bakry-\'{E}mery Ricci tensor is related to notion of quasi Einstein metrics. So, we hope Bakry-\'{E}mery Ricci tensors provide a good apparatus for geometrization of mass in the next section.
\section{New Field Equation}
In this section, we consider $(M,g)$ as a space-time manifold whose dimension is $n$ and $2\leq n$. Each smooth function $f$ on $M$ determines a 1-Bakry-\'{E}mery Ricci tensor on $M$. \\
Let $\mathcal{M}$ denotes the set of all 1-Bakry-\'{E}mery Ricci tensor on a manifold $M$ then, $\mathcal{M}$ can be written as the set of some pairs $(g, f)$, where $g$ is a semi-Riemannian metric and, $f$ is a smooth function on $M$. In fact, $\mathcal{M}$ can be consider as an open subset of an infinite dimensional vector space.\\
Denote canonical volume form of a meter $g$ on the oriented manifold $M$ by $dV_g$. It is well-known that critical points of Einstein-Hilbert functional 
\[\mathcal{L}(g)=\int_M R_g dV_g,\]
are metrics which satisfy the following equation,
\[\mathrm{Ric}-\dfrac{1}{2}Rg=0.\] 
The above equation is Einstein field equation in vacuum. It motivates to define a new Hilbert-Einstein functional $\mathcal{L}:\mathcal{M}\to \mathbb{R}$
\[ \mathcal{L}(g,f)=\int_M R_{(g,f)}dV_g,
\]
where, $R_{(g,f)}=\mathrm{tr}(\mathrm{Ric}_f^{1})$ is called scalar curvature of $\mathrm{Ric}_f^{1}$. Equations in which the critical points of this functional satisfy will be called the {\it new field equation}.\\
For a symmetric 2-covariant tensor $s$ on $M$, and a smooth function $h$ set,
\begin{align*}
\tilde{g}(t)&=g+ts\\
\tilde{f}(t)&=f+th
\end{align*}
For sufficiently small $t$, $(\tilde{g}(t) ,\tilde{f}(t))$ is a 1-Bakry-\'{E}mery Ricci tensor on M and is a variation of $(g, f)$. The Bakry-\'{E}mery Ricci tensor $(g, f)$ is a critical $1$-Bakry-\'{E}mery Ricci tensor for Hilbert action if and only if for any pair $(s,h)$:
\begin{equation}\label{eq}
\dfrac{d}{dt}\vert_{t=0} \mathcal{L}\big(\tilde{g}(t),\tilde{f}(t)\big)=
\dfrac{d}{dt}\vert_{t=0} \int_M R_{(\tilde{g}(t) ,\tilde{f}(t))} dV_{g+ts} =0,
\end{equation}
where, $R_{(\tilde{g}(t) ,\tilde{f}(t))}$ is the scalar curvature of $(\tilde{g}(t) ,\tilde{f}(t))$ and
\begin{equation*}
R_{(\tilde{g}(t) ,\tilde{f}(t))}=R_{\tilde{g}(t)}+\triangle_{\tilde{g}(t)} (\tilde{f}(t))-{|\vec{\nabla}\tilde{f}(t)|^2},
\end{equation*}
where, by notation $\vec{\nabla}$ we mean the gradient of smooth functions.\\
To find derivation in \ref{eq}, we must compute derivations of $R_{\tilde{g}(t)}$, $\triangle_{\tilde{g}(t)} (\tilde{f}(t))$, $|\vec{\nabla}\tilde{f}(t)|^2$ and
$dV_{g+ts}$ for $t = 0$. In \cite{bleeker}, it is shown
\begin{align}
(R_{\tilde{g}(t)})'(0)=&-\langle s,\mathrm{Ric}\rangle +\mathrm{div}(X),\qquad X\in \mathcal{X}(M),\\
(dV_{g+ts})'(0)=&\dfrac{1}{2}\langle g,s\rangle dV_g.
\end{align}
By local computations, we find derivation of $\triangle_{\tilde{g}(t)} (\tilde{f}(t))$ and $|\vec{\nabla}\tilde{f}(t)|^2$ at $t=0$. By means of a local coordinate system on $M$ with local frame $\partial_i$, we can write
\begin{align*}
\dfrac{d}{dt}\vert_{t=0} |\vec{\nabla}\tilde{f}(t)|^2 &=\big( \tilde{g}^{ij}(t) \dfrac{\partial \tilde{f}(t)}{\partial x^i} \dfrac{\partial \tilde{f}(t)}{\partial x^j}\big)' (0)=-s^{ij}\dfrac{\partial f}{\partial x^i}\dfrac{\partial f}{\partial x^j}+2g^{ij}\dfrac{\partial h}{\partial x^i}\dfrac{\partial f}{\partial x^j}\\
&=-\langle s, df\otimes df \rangle +2 \langle \vec{\nabla}h, \vec{\nabla}f \rangle .
\end{align*}
Local computation of Laplacian of a smooth function $f$ is as follows.
\[\triangle (f)= g^{ij} (\dfrac{\partial^2 f}{\partial x^i \partial x^j}-\Gamma_{ij}^k \dfrac{\partial f}{\partial x^k}).
\]
So,
\begin{align}
[\triangle_{\tilde{g}(t)} (\tilde{f}(t))]'(0) &= \Big( \tilde{g}^{ij}(t) (\dfrac{\partial^2 (f+th)}{\partial x^i \partial x^j}-\tilde{\Gamma}_{ij}^k (t)\dfrac{\partial (f+th)}{\partial x^k})\Big)'(0)\nonumber \\
&=-s^{ij} (\dfrac{\partial^2 f}{\partial x^i \partial x^j}-\Gamma_{ij}^k \dfrac{\partial f}{\partial x^k})+g^{ij}(\dfrac{\partial^2 h}{\partial x^i \partial x^j}-A_{ij}^k \dfrac{\partial f}{\partial x^k}-\Gamma_{ij}^k \dfrac{\partial h}{\partial x^k})\nonumber \\
&=-\langle s ,\mathrm{Hess}(f)\rangle +\triangle (f) -g^{ij}A_{ij}^k \dfrac{\partial f}{\partial x^k},
\end{align}
where, $A_{ij}^k =(\tilde{\Gamma}_{ij}^k)'(0)$. Define a vector field $Y$ by components $Y^k =g^{ij}A_{ij}^k$. We can write
\begin{align*}
Y^k &=g^{ij}A_{ij}^k =\dfrac{1}{2} g^{ij}g^{kl}(\nabla_i s_{jl}+\nabla_j s_{il}-\nabla_l s_{ij})\\
&=\dfrac{1}{2} g^{kl}(g^{ij}\nabla_i s_{jl}+g^{ij}\nabla_j s_{il}-g^{ij}\nabla_l s_{ij})\\
&=\dfrac{1}{2} g^{kl}(\mathrm{div}(s)_l +\mathrm{div}(s)_l -\nabla_l \mathrm{tr}(s))=\mathrm{div}(s)^k -\dfrac{1}{2}(\vec{\nabla} \mathrm{tr}(s))^k .
\end{align*}
Consequently,
\[ g^{ij}A_{ij}^k \dfrac{\partial f}{\partial x^k}= \mathrm{div}(s)^k \dfrac{\partial f}{\partial x^k} -\dfrac{1}{2}(\vec{\nabla} \mathrm{tr}(s))^k \dfrac{\partial f}{\partial x^k}=\mathrm{div}(s) (\vec{\nabla}f)-\dfrac{1}{2}\langle \vec{\nabla} \mathrm{tr}(s), \vec{\nabla}f \rangle.
\]
In the following, whenever it is convenient, we consider $s$ as a $(1,1)$ symmetric tensor. One can easily see
\[g^{ij}A_{ij}^k \dfrac{\partial f}{\partial x^k}=\mathrm{div}(s(\vec{\nabla}f))-\langle s,\mathrm{Hess}f\rangle-\dfrac{1}{2}\langle \vec{\nabla} \mathrm{tr}(s), \vec{\nabla}f \rangle .
\]
By above computations, we have
\begin{align}
[\triangle_{\tilde{g}(t)} (\tilde{f}(t))]'(0) &= \triangle (f)-\mathrm{div}(s(\vec{\nabla}f))+\dfrac{1}{2}\langle \vec{\nabla} \mathrm{tr}(s), \vec{\nabla}f \rangle .
\end{align}
Remind that the integral of divergence of every vector fields on $M$ is zero, consequently
integral of Laplacian of any smooth function on $M$ is zero. Also, for any two smooth
function $f$ and $h$ we have \cite{topics}
\[ \int_M \langle \vec{\nabla}h ,\vec{\nabla} f \rangle dV_g = -\int_M \triangle (f)h dV_g.
\]
Now, we are prepare to find critical $1$-Bakry-\'{E}mery Ricci tensor for Hilbert action.
\begin{align*}
0=\dfrac{d}{dt}\vert_{t=0} \mathcal{L}\big(\tilde{g}(t),\tilde{f}(t)\big)=&\dfrac{d}{dt}\vert_{t=0} \int_M R_{(\tilde{g}(t) ,\tilde{f}(t))} dV_{g+ts} \\
=&\int_M \big(R+\triangle(f)-{|\vec{\nabla}f|^2}\big)\dfrac{1}{2}\langle s,g \rangle dV_g\\
&+\int_M \big( -\langle s,\mathrm{Ric}\rangle +\mathrm{div}(X)+ \triangle (f)-\mathrm{div}(s(\vec{\nabla}f))\\
&+\dfrac{1}{2}\langle \vec{\nabla} \mathrm{tr}(s), \vec{\nabla}f \rangle +\langle s, {df\otimes df} \rangle -{2} \langle \vec{\nabla}h, \vec{\nabla}f \rangle\big)dV_g\\
=&\int_M \langle -\mathrm{Ric}+\dfrac{1}{2}Rg +{df\otimes df} -\dfrac{|\vec{\nabla}f|^2}{2}g,s \rangle dV_g \\
&+\int_M {2\triangle(f)}h dV_g
\end{align*}
The above expressions vanish for all pair $(g,f)$ if and only if
\begin{align}
\mathrm{Ric}-\dfrac{1}{2}Rg &= \big( df\otimes df-\dfrac{|\vec{\nabla}f|^2}{2}g \big),\label{field1}\\
\triangle(f)&=0. \label{field2}
\end{align}
In the case which $\mathrm{dim}M=4$, taking trace from both side of \ref{field1} yields $R={|\vec{\nabla}f|^2}$. The scalar curvature $R$ is related to matter distribution in points of the space-time manifold. Hence, $f$ is related to matter distribution. In fact, equation \ref{field2} express that $\mathrm{div}(\vec{\nabla}f)=0$ and, $\vec{\nabla}f$ can be interpreted as current of mass which satisfies conservation law. Also, we can interpret the expression appeared in the right hand side of \ref{field1} as stress-energy tensor of matter. We need to show that the divergence of 
\[ T^f :=df\otimes df-\dfrac{|\vec{\nabla}f|^2}{2}g,\] 
which retrieves conservation law of stress-energy tensor.
\begin{theorem}
Suppose that $f$ is a smooth function on a Riemannian manifold $(M,g)$ such that $\triangle (f)=0$, then the symmetric tensor $T^f$ is divergence-free.
\end{theorem}
\begin{proof}
Let $\{ e_i \}_{i=1}^n$ is an orthonormal base on $M$ and denote its reciprocal base by $\{ e^i\}_{i=1}^n$. So, we can write
\begin{align*}
\mathrm{div}(df\otimes df)(X)&=\sum_{i=1}^n (\nabla_{e_i} df\otimes df)(e^i ,X)\\
&=\sum_{i=1}^n \big( (\nabla_{e_i} df)\otimes df +df\otimes (\nabla_{e_i} df) \big) (e^i ,X)\\
&=\sum_{i=1}^n \big( (\nabla_{e_i} df)(e^i) df(X)+df(e^i)(\nabla_{e_i} df)(X)\\
&=\triangle(f)df(X) +\mathrm{Hess}f(\vec{\nabla}f ,X)=\mathrm{Hess}f(\vec{\nabla}f ,X).
\end{align*}
Also,
\begin{align*}
\mathrm{div}(|\vec{\nabla}f|^2g)(X)&=d(|\vec{\nabla}f|^2)(X)= X\langle \vec{\nabla}f, \vec{\nabla}f\rangle =2\langle \nabla_X (\vec{\nabla}f), \vec{\nabla}f\rangle \\
&=2(\nabla_X df)(\vec{\nabla}f)=2\mathrm{Hess}f(\vec{\nabla}f ,X).
\end{align*}
The above computations show that $\mathrm{div}(T^f)=0$.
\end{proof}
In the case which $\mathrm{dim}M\geq 3$, a simpler form of equation \ref{field1} is obtained.
\begin{theorem}
Suppose that $\mathrm{dim}M =n\geq 3$, then the critical Bakry-\'{Em}ery tensors of Hilbert action satisfy the following equations.
\begin{align}
\mathrm{Ric}&={df\otimes df}, \label{fild1}\\
\triangle (f)&=0.\label{fild2}
\end{align}
\end{theorem}
\begin{proof}
The second equation is the same as \ref{field2}. Let us assume that \ref{field1} holds, then by taking trace on both sides of this equation, we have
\[R-\dfrac{n}{2}R=(|\vec{\nabla}f|^2-\dfrac{n}{2}|\vec{\nabla}f|^2)\Rightarrow R={|\vec{\nabla}f|^2}.\]
By replacing the above formula in \ref{field1}, we derive \ref{fild1}. Now, let \ref{fild1} holds. By computation traces
of two sides of equation \ref{fild1}, we find 
\[R={|\vec{\nabla}f|^2}.\]
 Hence, by addition suitable expression to each side of equation \ref{fild1} we obtain \ref{field1}.
\end{proof}
Our obtained field equations are purely geometric when $(g,f,\lambda )$ be a 1-quasi Einstein metric for some $\lambda\in\mathbb{R}$. 
\begin{theorem}
Suppose that $(g,f)$ is a critic point of Hilbert-Einstein functional then, $(g,f,\lambda)$ is a 1-quasi Einstein metric if and only if $\lambda=0$ (the quasi-Einstein metric is steady) and $\vec{\nabla}f$ is a parallel vector field.
\end{theorem}
\begin{proof}
Assume $(g,f,\lambda)$ is 1-quasi Einstein metric on $M$, then
\[\mathrm{Ric}_f^1 =\mathrm{Ric}+\mathrm{Hess} f-df\otimes df=\lambda g.\] 
Hence, the equation \ref{fild1} implies that
\[\mathrm{Hess}f=\lambda g.\]
By taking traces on both sides of the above formula and using \ref{fild2}, we get
\[0=\lambda n,\]
so, $\lambda=0$ which means the quasi Einstein metric is steady and we have
\[\mathrm{Hess}f=0,\]
which implies that $\vec{\nabla}f$ is a parallel vector field on $M$. \\
Invesely, by definition of steady 1-quasi Einstein metric $(g,f)$ we can write
\[\mathrm{Ric}+\mathrm{Hess} f-df\otimes df=0.\] 
Since, $\vec{\nabla}f$ is a parallel vector field, we have $\mathrm{Hess}f=0$ which implies equations \ref{fild1} and \ref{fild2}.
\end{proof}
\textbf{Remark:} Denote by $\theta$ the 1-form $df$ appeared in above theorem. In the Riemannian manifold $(M,g)$
\[\mathrm{Hess}f=\nabla df=0,\]
means $\theta$ is a parallel 1-form. Moreover, in the generic case (where $\mathrm{Ric}$ is non degenerate) by the Weitzenb\"{o}ck formula we have $\theta =0$. Hence, in this case $f$ is locally constant and because of connectedness of $M$ it will be globally constant. So, generic Riemannian metrics are an obstruction to geometrize matter in general relativity.
\section{An example of space-time-mass manifold}
In this section we give an example of a manifold which satisfies the equations \ref{fild1} and \ref{fild2}. This example is an Einstein-de Sitter model in general relativity.\\
Suppose that $(N,\bar{g})$ is a Riemannian manifold of dimension $n\geq 2$ and, set $M = N \times (0,\infty)$. Any tangent vector to $M$ at a point $(p, t)$ is of the form $(v, \lambda)$ in which $v\in T_p N$ and $\lambda\in \mathbb{R}$ is a scalar. The vector field $(0, 1)$ on $M$ is denoted by $\partial_t$. Considering this vector field as a derivation of $C^\infty (M)$, for a smooth function $h(p, t)$ on $M$ we have $\partial_t (h)={\partial h / \partial t}$. Denote vector fields on $N$ by $X, Y , Z, \cdots$ and consider them as special vector fields on $M$. Also, denote the second projection map $(p, t)\mapsto t$ on $M$ by $t$. We can interpret $t$ as time. Note that for the 1-form $dt$ we have $dt(v, \lambda) = \lambda$, so $dt(\partial_t) = 1$.\\
For some smooth function $a:(0,\infty)\to \mathbb{R}$ define a metric $g$ on $M$ as follows.
\[g=e^{2a(t)}\bar{g}-dt\otimes dt.\]
Consider $f:N\times (0,\infty)\to \mathbb{R}$ which depends only on $t$ and denote it by $f(t)$. Consequently, $df=f'(t)dt$, $\vec{\nabla}f=-f'(t)\partial_t$ and $|\vec{\nabla}f|^2=-|f'(t)^2|$.\\
Let $\overline{\nabla}$ and $\nabla$ denote the Levi-Civita connections of $N$ and $M$, respectively. Routine computations show that
\begin{align*}
\nabla_X,Y &=\overline{\nabla}_X,Y +a'(t)g(X,Y)\partial_t , \\
\nabla_X,\partial_t &= a'(t)Y,\\
\nabla_{\partial_t} \partial_t &=0.
\end{align*}
Also, an easy local computation indicates that
\begin{equation}
\triangle(f) =-na'(t)f'(t)-f''(t).
\end{equation}
Hence, equation $\triangle(f) =0$ implies that $|f'(t)|=ce^{-na(t)}$, for some constant $c$.\\
Now, let $\overline{\mathrm{Ric}}$ and $\mathrm{Ric}$ denote Ricci curvature tensors of $N$ and $M$, respectively. Straightforward computations show that
\begin{align*}
\mathrm{Ric}(X,Y)&=\overline{\mathrm{Ric}}(X,Y)+(a''(t)+na'(t))g(X,Y), \\
\mathrm{Ric}(X,\partial_t)&=0,\\
\mathrm{Ric}(\partial_t ,\partial_t)&=-n(a''(t)+a'(t)).
\end{align*}
Equation \ref{fild1} for this example, becomes $\mathrm{Ric}=|f'(t)|^2 {dt\otimes dt}$. Hence, this equation holds if and only if
\begin{align}
\overline{\mathrm{Ric}}(X,Y)&=-(a''(t)+na'(t))g(X,Y),\label{ri}\\
n(a''(t)+a'(t))&={|f'(t)|^2}={c^2 e^{-2na(t)}}\label{ri1}
\end{align}
Left side of \ref{ri} does not depend on $t$, so $a''(t)+na'(t)$ have to be constant and $N$ is an Einstein manifold. In the case $a''(t)+na'(t)=0$ we find solution $a(t)={1/n}.\ln t$ and for this solution, equation \ref{ri1} also holds for $c=\sqrt{(n-1)/n}$.\\
So, for an $n$-dimensional Ricci flat manifold $(N,\bar{g})$ the meter $g=t^{{2}/{n}}\bar{g}-dt\otimes dt$ on
$M$ and the function $f(t)=\sqrt{(n-1)/n} \ln t$ satisfies field equations \ref{fild1} and \ref{fild2}. In this model, as $t$ approaches to zero, universe become smaller and density of matter increases to infinity. Time $t = 0$ is not in $M$ and this time is the instant of Big-Bang.

\end{document}